\documentclass[twoside,12pt, leqno]{amsart}
\usepackage{amsmath,amscd,amsthm,amssymb,amsxtra,latexsym,epsfig,epic,graphics}
\usepackage[all]{xypic}
\usepackage{graphicx}
\usepackage{MnSymbol}
\usepackage{color} 
\usepackage{booktabs}
\usepackage{multirow, color}

 \def\coker{{\rm coker}}
\DeclareMathOperator{\rank}{{\rm rank}}

\newtheorem{theorem}{Theorem}[section]

\newtheorem{proposition}[theorem]{Proposition}

\newtheorem{conjecture}[theorem]{Conjecture}

\theoremstyle{definition}
\newtheorem{definition}[theorem]{Definition}

\newtheorem{remark}[theorem]{Remark}

\newtheorem{example}[theorem]{Example}

\def\AA{{\mathbb A}}

\def\PP{{\mathbb P}}

\def\QQ{{\mathbb Q}}
\def\CC{{\mathbb C}}

\def\ZZ{{\mathbb Z}}

\def\sO{{\mathcal O}}

\DeclareMathOperator{\Proj}{{Proj}}
\DeclareMathOperator{\Spec}{{Spec}}

\def\Spec{{{\rm Spec}\,}}

\def\lbracket{{[\kern-1.5pt[}}
\def\rbracket{{]\kern-1.5pt]}}

\makeatletter
\def\Ddots{\mathinner{\mkern1mu\raise\p@
\vbox{\kern7\p@\hbox{.}}\mkern2mu
\raise4\p@\hbox{.}\mkern2mu\raise7\p@\hbox{.}\mkern1mu}}
\makeatother

\usepackage{times}
\newdimen\x \x=12pt

\usepackage{color}

\usepackage[breaklinks,bookmarksopen,bookmarksnumbered,urlcolor=blue]{hyperref}
\hypersetup{colorlinks=true,backref=true,citecolor=blue}

\author[David Eisenbud]{David Eisenbud}
\address{Department of Mathematics, University of California at Berkeley, Berkeley, CA 94720, USA}
\email{de@berkeley.edu}

\author[Frank-Olaf Schreyer]{Frank-Olaf Schreyer}
\address{Fakult\"at f\"ur Mathematik und Informatik, Universit\"at des Saarlandes, Campus E2 4, 66123 Saarbr\"ucken, Germany}
\email{schreyer@math.uni-sb.de}

\title{Minimal  Non-Weierstrass Semigroups}

\begin{document}

\begin{abstract}
Let $p\in X$ be a point on a compact Riemann surface. The \emph{Weierstrass} semigroup of $p$ is  the semigroup of pole orders
of meromorphic functions on $X$ that are regular at all but $p$. Hurwitz asked in 1892 
whether all numerical semigroups occur as Weierstrass semigroups. 

In this paper we give a new method  for showing that certain numerical semigroups are not Weierstrass, including some
of every genus $g$ in which non-Weierstrass examples could exist, except $g =18$. Our example for $g=13$ has at once the smallest possible genus, multiplicity
and number of generators of any possible non-Weierstrass semigroup.
 
 Our technique applies to prove that certain semigroups are not Weierstrass semigroups over any field. 
\end{abstract}
\maketitle

\section*{Introduction}

A \emph{numerical semigroup} is a subset of the non-negative integers
with finite complement that contains 0 and is closed under addition. To simplify the
language we generally drop the word ``Numerical''. 

If $S$ is a semigroup then the \emph{gaps} of $S$ are the positive integers
not in $S$ and the cardinality $g$ of the set of gaps is called
the \emph{genus} of the semigroup. The smallest nonzero element $m$ of $S$
is the \emph{multiplicity} and the minimal number of generators is the
\emph{ embedding dimension}  of the semigroup algebra
$k[t^{s} \mid s\in S]$, where $k$ is a field. The semigroup is called Weierstrass if
it occurs as the semigroup of pole orders of rational functions that are regular
except at a point $p$ of a compact Riemann surface $X$; by Weierstrass' L\"uckensatz (an easy consequence of the Riemann-Roch theorem) the genus
of $X$ coincides with the genus of the semigroup. 

The main contribution of this paper is a condition  on the free resolution of a semigroup ideal that forces every positively graded deformation to have a certain multivalued
section.  Under an additional hypothesis, this multisection tracks singular points of each fiber, proving
that no deformation can be smooth. By Pinkham's Theorem~\cite[Chapter IV]{MR376672}, this implies that the semigroup is not Weierstrass. We call a semigroup satisfying these conditions \emph{special} (Definition~\ref{special def}).

On the other hand, we show that every semigroup of genus $<13$ is Weierstrass. There are 1413 such semigroups \cite{NumericalSemigroupsSource}, or \url{https://oeis.org/A007323}. Results in the literature, quoted below, prove that 781 of them occur as
Weierstrass semigroups. In an auxiliary file (see below) we give explicit positively graded smoothing families
for the other 632, found using computational methods that we explain in Section~\ref{comp}. By
Pinkham's Theorem~\cite{MR376672}, this completes the proof that all 1413 examples occur
as Weierstrass semigroups.

Since every semigroup of genus $<13$ is Weierstrass, as is every semigroup of multiplicity $<6$ or embedding dimension $<4$ (see the history below),  our genus 13 example, which is generated by 
$$
\{6,9,13,16\}
$$
 has the smallest possible genus, multiplicity, and embedding dimension of any non-Weierstrass semigroup. In Theorem~\ref{all genera}, we exhibit 4-generator semigroups of multiplicity 6 and
 every genus $g \geq 13$ except $g=15$ and $g=18$. For $g=15$ we give an example of
 multiplicity 8.

We have put the Macaulay2 code necessary to construct and verify our assertions in the
GitHub repository {\bf eisenbud/WeierstrassSemigroups} \\ \url{https://github.com/eisenbud/WeierstrassSemigroups}.\\
In the file {\bf smoothingFamilies.dbm} we give smoothing families for the 632 semigroups that were not previously known to be Weierstrass.
To check the flatness and smoothness
of these families, simply execute the code in the file 
{\bf allSemigroupsOfGenusLessThan13AreWeierstrass.m2}.
The code that we have used, both to test whether a semigroup is degree-special and
to prove that certain families we have constructed are special, is in 
{\bf NumericalSemigroups.m2} \cite{NumericalSemigroupsSource} and {\bf WeierstrassSemigroups.m2} \cite{WeierstrassSemigroupsSource}. These packages
are distributed with the Macaulay2 system
\cite{M2}.

\subsection*{Background and history}
Suppose that $X$ is a compact Riemann surface, or more generally a smooth
projective curve over an algebraically closed field, of genus $g\geq 2$. Since there
are no everywhere regular analytic functions on $X$ except constants, the nonconstant
meromorphic functions
that are regular except at a point
$p\in X$, 
must have poles at $p$. Since the pole order of a product
of two such functions is the sum of their pole orders, and there are always functions
regular on $X\setminus \{p\}$ with
pole order exactly $2g +c$ for any $c \geq 0$, the pole orders form
a \emph{numerical semigroup} $S := S(X,p)$. 

In his lectures in the 1860's Weierstrass proved that, for every $p\in X$, 
the semigroup $S(X,p)$ contains all but exactly
$g :=$ genus $X$ positive integers, and is thus a numerical semigroup of genus $g$.
The result was published by Weierstrass'  
student Schottky in \cite{schottky1877}. Following Haure \cite{MR1508928}
 the semigroup $S(X,p)$ is now called the Weierstrass
 Semigroup of $X$ at $p$.
 
Hurwitz \cite{MR1510753} used Weierstrass points to prove that the automorphism group of a Riemann surface is finite, and posed the problem of deciding
which semigroups actually occur as Weierstrass semigroups. The problem attracted immediate attention: Haure  gave an incorrect criterion
for a semigroup to occur, excluding for example  the Weierstrass semigroup $\{4,6,19\}$ \cite[p, 165]{MR1508928}. Hensel-Landsberg \cite[pp. 545--548]{MR192045} took
up the problem, and gave an approach to constructing a Riemann surface with any given
semigroup using singular plane curves with properties that cannot always be achieved. (We are grateful to Jan Stevens for clarifying
this for us.)

There, it seems, the matter rested until 
Henry Pinkham \cite{MR376672}
gave an algorithm using deformation theory that can in principle determine whether a given
semigroup is Weierstrass (see Section \ref{Pinkham section}). 

Using deformations of curves of compact type, Eisenbud and Harris \cite{MR874034} proved
that in characteristic 0 all semigroups of low weight are Weierstrass. That
 result was extended to all characteristics by Osserman~\cite{MR2266887}
and to a wider class of semigroups by Pflueger~\cite{MR3892968}.
It was known that every semigroup having genus $<8$
(\cite{MR1310748}) or multiplicity  $<6$ \cite{MR1174911} or number of
generators $<4$, and every symmetric semigroup with $4$ generators,  is Weierstrass~\cite{MR655415}.

On the other hand, in~\cite{BuchweitzThese} in 1980, Buchweitz gave the first example of a numerical semigroup that is not
Weierstrass. It is generated by the elements
$$
\{13, 14, 15, 16, 17, 18, 20, 22, 23\}
$$
and has genus 16. (Buchweitz's proof: the number of integers
that can be expressed as the sum of a pair of gaps of $S$ is
46, while the dimension of the  space of quadratic differential forms
on a surface of genus 16
is only 45.) Since 1980 many other semigroups have been shown not to
be Weierstrass, both using Buchweitz's method and otherwise; see for example 
  \cite{MR4332500}, \cite{zbMATH07431143}, and
 \cite{MR3010539}.
  Examples by Komeda \cite{MR3010539} had, until now, the
 smallest known multiplicities, 8 and 12.
Until this paper, no non-Weierstrass semigroup of genus $<16$ or multiplicity $<8$ was known.

For more on the historical impact of the theory of Weierstrass points, see~\cite{MR2403373}

\subsection*{Notation}
We generally define a semigroup
$S$ by giving an ordered set of generators $\{s_{0}< \cdots <s_{u}\}$. Abusing notation, we
then write $S = \{s_{0}, \dots, s_{u}\}$. If $S$ is a semigroup, and $k$ is a field, then the \emph{semigroup algebra} of $S$ over $k$ is the
graded ring $k[S] :=  k[\{t^{s}\mid s\in S\}] \subset k[t]$. The \emph{semigroup ideal}
is the kernel of the natural map from the polynomial ring on variables corresponding to the
$t^{s_{i}}$ to $k[S]$.

When we write a semigroup ring as a quotient of a polynomial ring, we use the $\ZZ$-grading induced by the semigroup. We order the variables by increasing degree, and label them by the residue class of their degree modulo the multiplicity. Thus for example
the semigroup ring corresponding to the semigroup generated by $\{6,9,13,16\}$ over a field $k$ is written as a quotient of the polynomial ring
$k[x_0, x_3, x_1, x_4]$.

If $M$ is a map between free modules, or a matrix representing such a map, we write $I_j(M)$ for the ideal generated by the $j\times j$ minors ($=$ determinants of $j\times j$ submatrices) of $M$.

\section*{Acknowledgements}
We would not have made any of the discoveries in this paper without extensive experimentation using  Macaulay2 \cite{M2}; in addition, Macaulay2 code provides the computational proofs of the existence theorem as explained in Section~\ref{comp}. 

We are grateful to Jan Stevens for clarifying some of the history, and settling a question we had concerning Example \ref{higher genus}.
Nathan Ilten helped us use his program \cite{VersalDeformationsSource} to settle a particularly hard example in genus 12.

The work done here was begun in the very congenial atmosphere of the Commutative Algebra Program at MSRI in the Fall of 2022. The Macaulay2 project is supported by  NSF DMS 20-01206, and the first author is supported by NSF DMS 2001649.

MSC classes:	14H55, 14H20, 13D02, 13P20

\section{Pinkham's theorem and a structure theorem} \label{Pinkham section}

In this section we review two known results that we will need.

\subsection{Pinkham's Theorem}
 
 Pinkham  \cite[Theorem 13.9]{MR376672} gives a criterion for a semigroup to be Weierstrass:

\begin{proposition}[Pinkham] \label{full Pinkham}
Let $S$ be a numerical semigroup of genus $g$. $S$ is a Weierstrass semigroup if and only if there exists a one-parameter positively graded  smoothing of the semigroup  ring $k[S]$ of $S$.
\end{proposition}

By \emph{ a positively graded smoothing} we mean a positively graded algebra whose grading extends the grading induced by the
semigroup, and whose general fiber is a smooth affine curve. This is what Pinkham would call a smoothing with negative grading (Pinkham's sign comes from the natural grading on the vector space $T^1$ of first-order infinitesimal deformations). For the reader's convenience we give a proof:

\begin{proof}
Suppose $S$ is the Weierstrass semigroup of a point $p$ on a smooth projective curve $C$.
Consider the line bundle $\sO_{C}(p)$ and the graded ring
$$
A= \bigoplus_{i\ge 0} H^0(C,\sO_{C}(ip)).
$$
As a $k$-algebra, $A$ is generated by elements of degree $s_{i}$ corresponding to generators of the semigroup $S$ and the element 
$z \in H^0(C,\sO_{C}(p))$ corresponding to the constant function $1$.
Choosing appropriate generators we see that $A/z \cong k[S]$. Since $C = \Proj A$ is smooth the open subset, where $z\not=0$ is smooth, and since
$z$ is a non-zero divisor $A$ is a flat $k[z]$-algebra.

Conversely, if  
$$
\xymatrix{ \Spec k[S] \ar[r] \ar[d] & \Spec A \ar[d]\cr
                 \Spec k  \ar[r] & \Spec k[z] \cr }
 $$
 is a homogeneous one-parameter smoothing, then after base change we may assume that $\deg z=1$.
 Then $C=\Proj A$ is a smooth projective curve in a weighted projective space with a unique point $p$ at infinity, which has $S$ as Weierstrass semigroup. Note that $C$ is smooth because the general point of the special fiber of the affine curve $\Spec k[S]$  
 is smooth; the singular point corresponds to the irrelevant maximal ideal of $A$.
 \end{proof}
                 
All the flat, positively graded families of $k[z]$-algebras specializing to $k[S]$ are flat subfamilies
of the universal positively graded unfolding of $k[S]$. If there is a natural multisection
defined over the unfolding whose values are singular points, Pinkham's theorem
shows that $S$ is not Weierstrass. Conversely if there are singular points in every fiber then they define a multisection. Thus in proving that a semigroup is not Weierstrass, it is natural to look for a multisection. However there are also families with section where the section actually follows
the ``origin'', tracing out a smooth point---see Example~\ref{smooth section}.

\subsection{A structure theorem for resolutions}

Here is a case of  \cite[Theorem 3.1]{MR340240} and \cite{Bruns}, that we will use. It is closely related to the ``determinant of a complex'' discovered by Cayley~\cite{Cayley} and also exploited in~\cite{KnudsonMumford} and \cite{MacRae}.

\begin{proposition}\label{BE}
 Suppose that 
 \xymatrix{
 F_{0}&\ar[l]^{\psi_{1}} F_{1}&\ar[l]^{\psi_{2}}  F_{2}&\ar[l]^{\psi_{3}} F_{3}& \ar[l] 0
 }
 is the free resolution of a module of grade $\geq 2$ over
 a local or positively graded ring $R$.  Writing $r_{i} = \rank (\psi_{i})$,
  there is a commutative diagram
 $$
\xymatrix{
 \bigwedge^{r_{1}}F_{1} \cong \bigwedge^{r_{2}} F_{1}^{*} 
\ar[rr]^{\bigwedge^{r_{2}} \psi_{2}^{*}} \ar[rd]_{\bigwedge^{r_{1}}\psi_{1}} && \bigwedge^{r_{2}} F_{2}^{*} \cong \bigwedge^{r_{3}}F_{2}\cr
 &\bigwedge^{r_{1}}F_{0}\cong R \cong \bigwedge^{r_{3}}F_{3}\ar[ru]_{\bigwedge^{r_{3}}\psi_{3}}& \cr
} .
$$
In particular 
$$
I_{r_{2}}(\psi_{2}) = I_{r_{1}}(\psi_{1})I_{r_{3}}(\psi_{3})
$$
\end{proposition}

\section{Special resolutions and Non-Weierstrass Semigroups}

 The \emph{format} of a free complex
$$
\xymatrix{
F_{0}&\ar[l]F_{1}&\ar[l]  \cdots
}$$
is the list of integers $\{\rank\ F_{0}, \rank\ F_{1}, \cdots\}$ where
we suppress the ranks corresponding to any terms $F_{i} = 0$ in the 
complex. 

\begin{definition}\label{special format}
We will say that a minimal resolution 
$$
 *)\quad 
 \xymatrix{P &\ar[l]^{\phi_{1}} P^{6}& \ar[l]^{\phi_{2}}  P^{8}&\ar[l]^{\phi_{3}} P^{3}&\ar[l] 0}
 $$
 over a ring $P$
is \emph{special of format $\{1,6,8,3\}$} if it admits an acyclic subcomplex
$$
  **) 
 \xymatrix{\quad P&\ar[l]^{\phi_{1}'} P^{4} &\ar[l]^{\phi_{2}'}  P^{4}&\ar[l]^{\phi_{3}'} P &\ar[l] 0.}
  $$
 that is, term by term, a summand of $*)$. 
  
 In matrix terms, this means that 
there is
 a choice of generators of the free modules of the complex $*)$ such that:

\begin{enumerate} 

\item  The last 4 entries in 
 the first column of the matrix $\phi_{3}$, are zero as in the diagram below, where we show the column as the first column:
 $$
 \phi_{3}:\quad 
\begin{pmatrix}
*&*&*\\
*&*&*\\
*&*&*\\
*&*&*\\
0&*&*\\
0&*&*\\
0&*&*\\
0&*&*\\
\end{pmatrix}\, ,
$$
and $\phi_{3}'$ is the $4\times 1$ submatrix corresponding
to the first 4 elements in this column.
 
 \item We may divide $\phi_{2}$ into two $6\times 4$
matrices, $\phi_{2} = A\mid B$, such that
 the 5th and 6th rows of $A$ are zero as in the following diagram:
 $$
A: \quad
\begin{pmatrix}
 *&*&*&*\\
*&*&*&*\\
*&*&*&*\\
*&*&*&*\\
0&0&0&0\\
0&0&0&0\\
\end{pmatrix},
$$
and $\phi_{2}'$ is the $4\times 4$ matrix corresponding
  to the first 4 rows of $A$.
  
  \item With these choices, $\phi_{1}'$ is the matrix
  corresponding to the first 4 generators of $I$
  \end{enumerate}

  \end{definition}

To show that some semigroups with special resolution
of format $\{1,6,8,3\}$ are not Weierstrass,
we need conditions that ensure that the $\{1,4,4,1\}$ subcomplex  persists
in flat graded deformations.

\begin{definition}\label{special def}
A numerical semigroup $S$ is \emph{special} if every flat positively graded deformation
of the semigroup ring $k[t^{S}]$ has minimal free resolution that is special of format $\{1,6,8,3\}$.
\end{definition}

\begin{proposition}\label{the multisection}
With notation as above, if *) is special of format $\{1,6,8,3\}$ and the image of $\phi_{1}$
is a prime ideal then, with notation as above, the four
entries of the matrix $\phi_{3}'$ are a regular sequence, and thus define a
codimension 4 subscheme of $\Spec P$.
Thus if the semigroup $S$ is special, then the entries of $\phi_{3}'$ define a finite multisection
 in every deformation of $\Spec k[S]$.
  \end{proposition}
  
\begin{proof} Let $J$ be the ideal of $P$ generated by the 4 entries of $\phi_{3}'$. The cokernel of the dual of $\phi_{3}$ is, up to a shift in grading, the canonical module $\omega$ of 
 $k[S]$, and thus its annihilator in $P$ is the annihilator $I$ of $k[S]$, which is prime. The
 ideal $J$ is the ideal quotient of one of the generators of $\omega$ into the other two,
 and thus $I\subset J$. Since $I$ requires 6 generators, $I\neq J$, and thus $J$ has
 codimension $4$. Since the resolution *) is minimal, the 4 generators of $J$ are of positive
 degree, and it follows that they form a regular sequence, proving the first statement. The second statement follows at once.
\end{proof}

 The definition of special requires information about every flat positively graded
 deformation of the semigroup ring, but we can give a sufficient condition on the minimal graded resolution
 of the semigroup ring alone:
 
In the matrix representing a homogeneous map of free modules
$\oplus P(c_{j}) \to \oplus P(r_{i})$  we say that the element in position $(i,j)$
 has \emph{formal degree} $c_j-r_{i}$. If the element is nonzero,
 this is the ordinary degree. In our case $P$ will be 
 positively graded, so if an element of a matrix
has  negative formal degree, then the element is zero.

\begin{proposition}\label{stability of conditions}
Let $S$ be a numerical semigroup, and write the semigroup ring $k[t^{S}]$ of $S$ as $P/I$, where $P$ is a graded polynomial ring 
whose generators have degrees corresponding to the generators of $S$.
If the minimal graded $P$-free resolution of the semigroup ring $k[t^{S}] = P/I$
  has the form
  $$
 \quad 
 \xymatrix{P&\ar[l]^{\phi_{1}} P^{6}&\ar[l]^{\phi_{2}}  P^{8}&\ar[l]^{\phi_{3}} P^{3} &\ar[l] 0, }
 $$
and (possibly after reordering the generators of the free modules)
satisfies the following conditions, then $S$ is special.
\begin{enumerate}
 \item The last four entries of the first column of the matrix $\phi_{3}$ 
 have negative formal degree.
 \item The first three entries of each of the last two rows of $\phi_{2}$
 have formal degree strictly less than the multiplicity of
 $S$.
 \item The induced subcomplex \\
  \centerline{
 \quad \xymatrix {P &\ar[l]^{\phi_{1}'} P^{4} &\ar[l]^{\phi_{2}'}  P^{4}&\ar[l]^{\phi_{3}'} P &\ar[l] 0, }
 }
 where $\phi_{1}'$ consists of the first 4 columns of $\phi_{1}$,
 $\phi_{2}'$ consists of the first 4 rows and columns of $\phi_{2}$
  and $\phi_{3}'$
 consists of the first four rows of the first column of $\phi_{3}$,
 is acyclic.
\end{enumerate}
Moreover, when these conditions are met, the 4 elements of the matrix $\phi_{3}'$
form a regular sequence.
\end{proposition}

\begin{proof} [Proof of Proposition~\ref{stability of conditions}]
 The elements of negative formal degree are of course 0. 
 
 Write $\omega := \coker \phi_{3}^{*}$. Since $\omega$ is, up to a shift in
grading, the canonical module of 
$P/I$, the annihilator of $\omega$ is $I$. Let $w_{i}$ be the generator of $\omega$
corresponding to the $i$-th column of $\phi_{3}$.
The ideal $J$ generated by the first column of $\phi_{3}$ is the annihilator
of the image of $w_{1}$ in $\omega/(Pw_{2}+Pw_{3})$, so $I\subset J$. By
the special hypothesis, $J$ is generated by 4 elements, whereas 
$I$ requires 6 generators, so $I\subsetneq J$. As $I$ is a prime of codimension 3
$J$ must have codimension 4, so the 4 nonzero elements in the first column of $\phi_{3}$
are a regular sequence, as claimed.

The degrees of these four elements are in the semigroup $S$, and since
they form a regular sequence, any nonzero relation
among them must have degrees in $S$ too.  
The first 4 elements in each row of
$\phi_{2}$ form such a relation.
By the special
hypothesis, the first 3 elements of the last two rows of $\phi_{2}$ are zero,
and since they are relations on a regular sequence the 4th element in each of these rows must be zero as well.
\end{proof}

\begin{definition}\label{degree-special def} We call a semigroup with format $\{1, 6, 8, 3\}$ \emph{degree-special}
if it satisfies the conditions (1), (2) and (3) of Proposition \ref{stability of conditions}; this implies that the semigroup is special.
\end{definition}

The following conjecture is based on our examination of hundreds of examples:
\begin{conjecture}\label{acyclity}
If $S$ satisfies conditions (1) and (2), then it satisfies condition (3) as well; so 
$S$ is special without further hypotheses.
\end{conjecture}

\begin{example} The semigroup $S$ generated by $\{6,9,13,16\}$ is degree-special.
 Indeed, the semigroup ring of  $S$ has differentials
 
\begin{small}
$$  
\begin{aligned}
\phi_{1}& = 
\begin{pmatrix}
 x_0^3-x_3^2 &x_3x_1-x_0x_4 &x_0^2x_1-x_3x_4 &x_0x_1^2-x_4^2 &x_0^2x_3^3-x_1^3& x_0x_3^4-x_1^2x_4 
\end{pmatrix}
\\
\phi_{2} &= 
\begin{pmatrix}
  -x_1& -x_4&   0&    0 &      0 &     -x_0x_3^3 &0  &       x_0x_3^2x_4 \\
-x_3 &-x_0^2& -x_4& -x_0x_1 &-x_1^2 &0    &     -x_0x_3^3& x_0^3x_3^2  \\
x_0 & x_3  &  -x_1 &-x_4   & 0  &    x_1^2  &   0   &      0           \\
0   & 0    &  x_0&  x_3 &    0  &    0     &    0      &   -x_1^2      \\
0  &  0 &     0 &   0   &    -x_3  & x_0^2   &  -x_4  &    -x_0x_1     \\
0  &  0  &    0   & 0&       x_0 &   -x_3 &     x_1 &      x_4         
\end{pmatrix}
\\
\phi_{3} &= 
\begin{pmatrix}
x_4 & x_0x_3^3  &  0      \\ 
 -x_1 &-x_0^2x_3^2 &0      \\
 -x_3 &-x_1^2  &    0      \\
 x_0  &0   &        -x_1^2 \\
 0   & x_4   &      x_0x_1 \\
 0   & -x_1   &     -x_4   \\
 0    &-x_3    &    -x_0^2\\
 0   & -x_0    &    -x_3   
\end{pmatrix}
\end{aligned}
$$
\end{small}
and satisfies the degree conditions.
 \end{example}

\begin{proposition}\label{1441}
 Suppose that $P$ is a local or positively graded ring, and that
 $$
 \xymatrix{P &\ar[l]^{\psi_{1}} P^{4}&\ar[l]^{\psi_{2}} P^{4} &\ar[l]^{\psi_{3}} P&\ar[l] 0}
 $$
 is a minimal free resolution of format $\{1,4,4,1\}$. If the 4 entries of the matrix $\psi_{3}$ form a 
 regular sequence, then the entries of the matrix $\psi_{2}$
 are in the ideal $J$ generated by the entries of $\psi_{3}$, 
 and the entries of $\psi_{1}$ are
 in $J^{2}$.
\end{proposition}

\begin{proof} 
Since $I:= I_1(\psi_1)$ is an ideal of finite homological dimension, \cite[Corollary 5.2]{MR340240} shows that
we may write $I = rI'$ for some nonzerodivisor $r$ and
some ideal $I'$ of depth $\geq 2$. Thus, replacing $\psi_{1}$
by $\psi_1' := r^{-1}\psi_{1}$ if necessary, we may assume that $*)$ is a resolution
of a module of grade $\geq 2$.


The first assertion of the Proposition holds because the columns of  $\psi_{2}^{*}$ are
syzygies of the regular sequence of elements of $\psi_{3}^{*}$, 
and are thus linear combinations of the Koszul syzygies. 

By  Proposition~\ref{BE} we have
$I_{3}(\psi_2)= I_{1}(\psi_1') J$. By the first assertion, $ I_{3}(\psi_2)\subset J^{3}$. Thus
$$
I \subset I_{1}(\psi'_1) \subset I_{3}(\psi_2):J\subset J^{3}:J.
$$
Since $J$ is a complete intersection,
 $$
 J^3:J = J^{2},
 $$
 by \cite[Exercise 17.13 e)]{MR1322960}
 completing the proof. (The last equality also follows directly from the well-known fact that the Rees algebra
 of the ideal generated by a regular sequence is isomorphic to a polynomial ring.) 
 \end{proof}

\begin{theorem}\label{main}
 If $S$ is a special semigroup  then $S$ is not a Weierstrass semigroup.
\end{theorem}

\begin{proof}  We will show that any graded deformation
of $\Spec k[t^{S}]$ is singular along the distinguished multisection
of Definition~\ref{special def}. Proposition~\ref{full Pinkham}
shows that $S$ is not Weierstrass.

With notation as in the definition of special,
we consider the minimal free resolution of the semigroup ring $P/I$. 
The condition that the last four entries of the first column of $\phi_{3}$ have negative formal degree
persists in any flat deformation. By Proposition~\ref{stability of conditions}
the first four entries of that column form a regular sequence, and this condition too persists.
Also the degree condition on the first four entries of the last two rows of $\phi_{2}$
persists, and implies that these entries are zero. Thus there is a $\{1,4,4,1\}$ subcomplex
in any deformation.  By condition (3) of Proposition~\ref{stability of conditions}, the subcomplex
of format $\{1,4,4,1\}$ in the resolution of the semigroup ring is acyclic. Since acyclicity is an open condition, this persists in every deformation.

To complete the proof we will show,
 more generally,
  that if $P$ is a positively graded polynomial ring, and $P/I$ is a 
 factor ring with minimal free resolution of special format $\{1,6,8,3\}$
 then with ideals $I_{1}(\psi_1), I_{3}(\psi_2), J = I_1(\psi_3)$ defined in terms of the 
 distinguished $\{1,4,4,1\}$ subcomplex as in  the proof of Proposition~\ref{1441}, 
  $\Spec P/I$ is singular along $\Spec P/J$.
  
This follows because
 $I_{1}(\psi_1)$ is generated by four of the 6 generators of $I$ and lies in $J^{2}$.
 Thus all partial derivatives of the polynomials generating $I_{1}(\psi_1)$ are in  $J$, so
 the $3\times 3$ minors 
 of the Jacobian matrix of $I$ also lie in $J$ proving the claim. \end{proof}

\begin{example}\label{good 4x4}
Returning to the example of the semigroup generated by $\{6,9,13,16\}$ above we see that the
induced subcomplex of format $\{1,4,4,1\}$ has differentials
$$
\begin{aligned}
\phi_{1}' &= \quad
\begin{pmatrix}
 x_0^3-x_3^2 &x_3x_1-x_0x_4 &x_0^2x_1-x_3x_4 &x_0x_1^2-x_4^2 
 \end{pmatrix}
\\
 \phi_{2}' &= \quad
\begin{pmatrix}
-x_1 &-x_4  & 0  &  0    \\   
-x_3 &-x_0^2 -x_4& -x_0x_1 &-x_1^2 \\
x_0  &x_3   & -x_1 &-x_4 \\
0   & 0   &   x_0 & x_3 \\   
\end{pmatrix},\ 
\quad
\phi'_{3}= \quad
 \begin{pmatrix}
x_4   \\
-x_1 \\
 -x_3   \\
x_0  \\
\end{pmatrix}.
\end{aligned}
$$
Using the theorem of \cite{MR314819} it is easy to show that this subcomplex
is acyclic. In particular $S$ is not a Weierstrass semigroup.
\end{example}

There are several ways to check
condition (3), the acyclicity of the $\{1,4,4,1\}$ subcomplex, in the definition of degree-special:

\begin{proposition}\label{degree-special consequences}
Suppose that $S$ is a numerical semigroup whose minimal free resolution satisfies
conditions (1) and (2) in Definition~\ref{degree-special def}, and let
$$
\xymatrix{P &\ar[l]^{\phi_{1}'} P^{4} &\ar[l]^{\phi_{2}'} P^{4} &\ar[l]^{\phi_{3}'} P &\ar[l] 0}
$$
 be the subcomplex of the minimal free resolution of $P/I = k[t^S]$
 as in the proof of Proposition~\ref{stability of conditions}. Each of the following 
 hypotheses implies condition (3) of Proposition~\ref{stability of conditions}.
 
\begin{enumerate}
\item \label{BE hypothesis} The $3\times 3$ minors of $\phi_{2}'$ generate an ideal of codimension $\geq 2$.

 \item\label{2} The $2\times 2$ minors of the $4\times 2$ submatrix of $\phi_{3}$ consisting of the 
 last 4 rows of the last 2 columns generate an ideal of codimension $\geq 2$.
 
 \item \label {3} After possibly reordering the rows and columns of $\phi_{2}'$, the $\phi_{i}'$ satisfy the following 3 conditions, 
(compare Example~\ref{good 4x4}):

\begin{enumerate}

 \item 
 The first two elements of the last row of $\phi_{2}'$ are zero. 
 
 \item the last two elements of the last row of $\phi_{2}'$ form a regular sequence.

 \item  the $2\times 2$ minors of the first two columns of $\phi_{2}'$
 are nonzero and have formal degrees equal to the formal degrees of
 the first 3 generators of the semigroup ideal.
  
\end{enumerate}
  
\end{enumerate}
\end{proposition}

\begin{proof}
The sufficiency of Hypothesis~(\ref{BE hypothesis}) follows from the main theorem of~\cite{MR314819}.

\noindent{Hypothesis~(\ref{2}) $\Rightarrow$ Hypothesis~(\ref{BE hypothesis})}:   
As shown in the diagram below, we divide the three matrices 
 $\phi_{1},\phi_{2}, \phi_{3}$ into blocks. In the diagram below $A$ is a $1\times 4$ block, so $A^t$, the transpose, is a $4\times 1$ block.
 Similarly, $B = \phi_2'$ has size $4\times 4$, $\widetilde B$ has size $2\times 4$, $C^t$ has size $1\times 4$ and $\widetilde C^t$ has size $2\times 4$.
 The remaining blocks have entries 0 or entries $*$ that do not play a role in the argument below.
  $$ 
  \begin{array}{|c|c|} \hline
               \phi_{1}^{t}& \phi_{2} \cr \hline
                                 & \phi_{3}^{t} \cr \hline
       \end{array} =
\begin{array}{|c||cccc| cccc|}\hline
  
          &&&&&*&*&*&*\cr        
          A^t&&&B&&*&*&*&*\cr        
          &&&&&*&*&*&*\cr        
          &&&&&*&*&*&*\cr\hline
         *               & {0}&0&0&0&&&&\cr
       *                    & {0}&0&0&0&&\widetilde B&&\cr\hline\hline
                                                          & &&C^t&&0&0&0&0\cr\hline
                                                         
                                                          &*&*&*&*&&&&\cr
                                                          &*&*&*&*&&\widetilde C^t&&\cr\hline
                                                          \end{array} $$
By Proposition~\ref{BE} we have
$$
I_3(B)I_2(\widetilde B) = I_1(A)I_1(C)I_2(\widetilde C).
$$
Under Hypothesis (\ref{2}), $I_2(\widetilde C)$ has depth $\geq 2$. Since $I$ is prime, every minimal generator of $I$ is prime, and thus $I_1(A)$
has depth $\geq 2$. Moreover, $I_1(C)$ has depth 4, and thus $I_3(B)\supset I_3(B)I_2(\widetilde B) $ has depth $\geq 2$
completing the proof in this case.
 \smallbreak
 
 \noindent Hypothesis (\ref{3}) $\Rightarrow$ Hypothesis (\ref{BE hypothesis}):   We now subdivide the $4\times 4$ matrix $B = \phi_2'$ into 4 blocks
 $$
 B = \begin{array}{|cc|cc|} \hline
               &&*&* \cr 
              B_1&&*&* \cr 
              &&*&* \cr\hline
               0&0&B_2&\cr\hline
       \end{array} 
       $$
 We see that $I_3(B) \supseteq I_2(B_1)I_1(B_2)$. Since we have assumed that $I_1(B_2)$ has depth 2, it suffices to
 prove the same for $I_2(B_1)$

 Let $B_1$ be the $3\times 2$ submatrix of $B$
 consisting of the first two columns and the first 3 rows. By assumption the $2\times 2$ minors
 of $B_1$ are not all zero, and thus generate an ideal of depth $\geq 1$. Since each minimal generator of $I$ is prime,
the first three generators of $I$ define an ideal of depth $\geq 2$. It follows from the main theorem of~\cite{MR314819} that the
image of the column matrix made of these three generators is the kernel of the transpose of $B_1^t$.

On the other hand, Cramer's rule implies that the column matrix of the $2\times 2$ minors of $B_1$ composes with $B_1$ to 0,
and thus it is a multiple of the column of the first three generators of $I$. It follows that the $2\times 2$ minors of $A$ are
 multiples of these generators, and since they have the same degrees, they
 must generate the same ideal, completing the proof.
 \end{proof}

\section{Degree-special semigroups of larger genus}

The following result covers all genera $\geq 13$ except  $g=18$.  Direct enumeration shows that
no degree-special semigroup has genus 18 (we have no conceptual understanding of this gap).

\begin{theorem}\label{all genera}
 For $g\geq 13\hbox{ but } g \not \equiv 0  ({\rm mod}\ 3),$ the semigroup $\{6,9,g,g+3\}$  is degree-special of genus $g$.  For 
$g\geq 21 \hbox{ and } g \equiv 0({\rm mod}\ 3),$ the semigroup $\{6,15,g-1,g+2\}$
is degree-special of genus $g$. The semigroup $\{8,10,13,15\}$ is degree-special of genus 15.
\end{theorem}

\begin{proof} We give the details for the case of the semigroups $\{6,9, g, g+3\}$ with 
$g\equiv 2 \hbox{ or } 4({\rm mod}\ 6)$, the other cases being similar.

The semigroup ideal  of $L$ has a $\QQ[x_{0},x_{3},y,w]$-resolution
$$\xymatrix{ 
0 & \ar[l] S & \ar[l]_{\phi_{1}} S^{6} & \ar[l]_{\phi_{2}} S^{8} & \ar[l]_{\phi_{3}} S^{3} &\ar[l] 0 \cr}.
$$
In case of $g\equiv 2 \hbox{ or } 4 ({\rm mod}\ 6)$ the differentials are as indicated in Table~1.

 \begin{table}\label{syzygytable}
$$ \begin{array}{|c|c|} \hline
               \phi_{1}^{t}& \phi_{2} \cr \hline
                                 & \phi_{3}^{t} \cr \hline
       \end{array}=
\begin{array}{|l|cccc cccc|}\hline
       {\color{red}x_{3}^{2}}-x_{0}^{3} & -y&-w&0&0&0&f&0&0\cr
       {\color{red}x_{3}y}-w\,x_{0}        &  {\color{red}x_{3}}&x_{0}^{2}&-w&y\,x_{0}&-y^{2}&0&-f&0\cr
       {\color{red}x_{3}w}-y\,x_{0}^{2}  &  x_{0}& {\color{red}x_{3}}& {\color{red}y}&-w&0&-y^{2}&0&f\cr     
       {\color{red}w^{2}}-y^{2}x_{0}      &  0&0&-x_{0}& {\color{red}x_{3}}&0&0&0&-y^{2}\cr
       {\color{red}y^{3}}-x_{0}f                     & {\color{blue}0}&{\color{blue}0}&{\color{blue}0}&{\color{blue}0}&{\color{red}x_{3}}&-x_{0}^{2}&{\color{red}w}&-y\,x_{0}\cr
              {\color{red}y^{2}w}-x_{3}f                 & {\color{blue}0}&{\color{blue}0}&{\color{blue}0}&{\color{blue}0}&-x_{0}& {\color{red}x_{3}}& -y&{\color{red}w}\cr\hline
                                                          & w&-y& {\color{red}x_{3}}&x_{0}&{\color{blue}0}&{\color{blue}0}&{\color{blue}0}&{\color{blue}0}\cr
                                                          &f&0&y^{2}&0&-w&y&{\color{red}x_{3}}&-x_{0}\cr
                                                          & 0&-f&0&y^{2}&yx_{0}&-w&x_{0}^{2}&{\color{red}x_{3}}\cr\hline
                                                          \end{array} $$
                  \noindent       
where   $f \in k[x_{0}] $. 
\caption{Syzygy matrices for semigroups $\{6,9,g,g+3\}$ for $g\equiv 2 \hbox{ or } 4({\rm mod}\ 6)$.}
 \end{table}

Here we could use any monomial order on $S$ for which the terms colored in {\color{red}red} in the first matrix are the lead terms (for example we could take the weight order of the grading on $S$ refined by the reversed lexicographic order with $x_{3}>y>w >x_{0}$). The red terms in the syzygy matrices are then the lead term with respect to the induced monomial orders, and the complex is exact by
 the Gr\"obner basis algorithm for syzygies as presented for example in \cite{Schreyer25}[Section 8.3].  In this case the exactness
 of the $\{1,4,4,1\}$ subcomplex follows as well. 
In the universal family the entries indicated in {\color{blue}blue} remain zero for degree reasons if $g\ge 13$.
\end{proof}

Table 2 gives the complete list
of the 4-tuples of generators of degree-special semigroups up to genus 20:
\begin{table}[h]\label{table1}
\centering
\begin{tabular}{cl}
\toprule
Genus & Generators \\
\midrule
13 & $\{6, 9, 13, 16\}$, $\{8, 9, 12, 13\}$ \\
14 & $\{6, 9, 14, 17\}$ \\
15 & $\{8, 10, 13, 15\}$ \\
16 & $\{6, 9, 16, 19\}$, $\{8, 11, 12, 15\}$ \\
17 & $\{6, 9, 17, 20\}$, $\{9, 10, 14, 15\}$ \\
19 & $\{6, 9, 19, 22\}$, $\{8, 12, 13, 17\}$ \\
20 & $\{6, 9, 20, 23\}$, $\{6, 15, 19, 22\}$ \\
\bottomrule
\end{tabular}
\caption{All the degree-special semigroups of genus $g \leq 20$. 
}
\end{table}

\begin{example}\label{smooth section}
 As explained in the introduction, we could search for semigroups that are not Weierstrass by looking for those moving in families with a multisection. However, there are positively graded deformations of semigroup rings that admit a section, but where the image of the section
consists of smooth points of the fibers. This is the case for the semigroup
$L=\{7, 15, 23, 39\}$.
The genus of the semigroup is $30$. The minimal free resolution over $S=\QQ[x_{0},x_{1},x_{2},x_{4}]$ has shape
$$\xymatrix{ 
0 & \ar[l] S & \ar[l]_{\phi_{1}} S^{4} & \ar[l]_{\phi_{2}} S^{6} & \ar[l]_{\phi_{3}} S^{3} &\ar[l] 0 \cr}
$$
with differentials as indicated below.
\begin{tiny}
$$\begin{array}{|c|c|} \hline
               \phi_{1}^{t}& \phi_{2} \cr \hline
                                 & \phi_{3}^{t} \cr \hline
       \end{array}=
       \begin{array}{|c|ccc ccc|}\hline
{\color{red}x_{1}^{2}}-x_{0}x_{2}       & -x_{2}^{2}+x_{0}x_{4}&-x_{2}x_{4}&-x_{0}^{10}&-x_{4}^{2}+x_{0}^{9}x_{1}&x_{0}^{9}x_{2}&0\\
{\color{red}x_{2}^{2}}-x_{0}x_{4}        & {\color{red}x_{1}^{2}}-x_{0}x_{2}&-x_{0}x_{4}&-x_{1}x_{4}&0&x_{0}^{10}&x_{0}^{9}x_{1}-x_{4}^{2}\\
{\color{red}x_{1}x_{2}x_{4}}-x_{0}^{11}&  0&{\color{red}x_{1}}&{\color{red}x_{2}}&0&-x_{4}&0\\
{\color{red}x_{4}^{2}}-x_{0}^{9}x_{1} & 0  &-x_{0}^{2}&-x_{0}x_{1}&{\color{red}x_{1}^{2}}-x_{0}x_{2}&{\color{red}x_{1}x_{2}}&{\color{red}x_{2}^{2}}-x_{0}x_{4} \\ \hline
                                                         & x_{4}&-x_{2}&{\color{red}x_{1}}&x_{0}&{\color{blue}0}&{\color{blue}0}\\
                                                         &0&x_{4}&0&-x_{2}&{\color{red}x_{1}}&-x_{0}\\
                                                         &-x_{0}^{9}&0&-x_{4}&0&-x_{2}&{\color{red}x_{1}}\\ \hline
       \end{array}$$
\end{tiny}
Again the lead terms in suitable monomial orders are indicated in {\color{red}red}.
The universal unfolding has a section because the entries indicated by {\color{blue}blue} remain zero for degree reasons.
However this is a Weierstrass semigroup by \cite{Waldi80}, since the semigroup ring is an almost complete intersection.
Indeed a flat smoothing family is given by
    the polynomials
   $$\!\begin{array}{l}
      x_{1}^{2}-x_{2}x_{0}+x_{2}z^{7}\\
      x_{2}^{2}-x_{4}x_{0}\\
      x_{1}x_{2}x_{4}-x_{0}^{11}+x_{0}^{10}z^{7}-x_{0}^{2}z^{63}+x_{0}z^{70}\\
      x_{4}^{2}-x_{1}x_{0}^{9}-x_{1}z^{63}.
      \end{array}$$
In this family, the value of the section is always the point defined by the ideal $(x_{0},x_{1},x_{2},x_{4})$.
Computational steps to produce such families are described in Section \ref{comp}.
 \medbreak
\end{example}

\section{Further examples}

\begin{example}\label{higher genus} 
Let $L=\{6,9,13,16\}$ and consider the semigroup generated by $L'=2L \cup \{6+9\}=\{12,15,18,26,32\}$. The semigroup ideal $I_{L'}\subset S=\QQ[x_{0},x_{3},x_{6},x_{2},x_{8}]$ is generated by the image of $I_{L} \subset \QQ[x_{0},x_{3},x_{1},x_{4}]$ under the map $x_{i} \mapsto x_{2i}$ and the polynomial $f=x_{3}^{2}-x_{0}x_{6}$. Thus the resolution is the mapping cone of the endomorphism given by multiplication with $f$.
$$
\xymatrix{ S & \ar[l] S^{6} & \ar[l] S^{8}& \ar[l] S^{3}& \ar[l] 0 \\
\ar[u]^{f} S & \ar[u]^{f} \ar[l] S^{6} & \ar[u]^{f} \ar[l] S^{8}& \ar[u]^{f} \ar[l] S^{3}& \ar[l] 0 \\}
$$
For degree reasons, the universal unfolding has a 2-valued section and  the flat family will have an exact subcomplex of type
$$
\xymatrix{ S & \ar[l] S^{5} & \ar[l] S^{8}& \ar[l] S^{5} &\ar[l] S^{1} & \ar[l] 0. \\
}
$$
Jan Stevens (personal communication) showed that $L'$, like $L =\{6,9,13,16\}$, defines a smoothable curve singularity that is not 
smoothable in a positively graded family, and is thus not Weierstrass; indeed the 2-valued section defines a pair of double points in the general positively graded family.  Similar constructions yield a lot of Weierstrass semigroups whose families have a multivalued section. 
\end{example}

The previous example is closely related to Torres' construction~\cite{Torres94} of non-Weierstrass semigroups, of which the following is a straightforward extension. For the reader's convenience we give a proof.  

\begin{theorem}\label{torres}
Let $L$ be a non-Weierstrass semigroup and let $d$ be a prime. Any semigroup $L'$ generated by dL and a set
$X$ of  sufficiently large integers such that $d$ does not divide every $x\in X$ is  a non-Weierstrass semigroup.
\end{theorem}

\begin{proof} Suppose $L'$ is the Weierstrass semigroup of a point $p$ in a smooth projective curve $\widetilde C$.
Let $n$ be the maximal element among the minimal generators of $L$. We assume that the new generators of $L'$ are larger than $dn$. Consider the morphism $\varphi_{|dn p|}\colon \widetilde C \to \PP^{r}$ where $r=|\{\ell \in L \mid \ell<n\}|$.

 If $\varphi_{|dnp|}$ were birational onto its image then the genus of $\widetilde C$ would be limited by Castelnuovo's bound for the genus of non-degenerate curves of degree $dn$ in $\PP^{r}$ \cite[Theorem 10.4]{MR4898511}. But if the integers in $X$
are sufficiently large, then the genus of the semigroup $L'$ exceeds this bound. Thus the morphism $\varphi_{|dnp|}$ must factor 
through a morphism $\psi$  to a smooth projective curve $C$. Since $d = \gcd(dL)$ is prime,
the degree of $\psi$ would be $d$, and $p$ would be a point of total ramification.
Thus  $\psi (p)\in C$ would have $L$ as its Weierstrass semigroup, a contradiction.
\end{proof}

\begin{example}\label{Torres example} For a specific instance of Theorem~\ref{torres} we may take the semigroup $L$ to be $\{6,9,13,16\}$, which has genus $13$ and, with $d = 2$, we take $L'=\{12,18,26,32, 2x+1\}$ with $x>55$. By \cite[Theorem 10.4]{MR4898511}, 
the genus of a non-degenerate curve of degree $dn=32$
in $\PP^{r}$  is bounded by
 $$\pi_{0}(dn,r)=(r-1)M(M-1)/2+M\epsilon$$
where $M = \lfloor(dn-1)/(r-1)\rfloor$ and $\epsilon = (dn-1) - (r-1)M$. 
Taking $r = 6, d=2, n=16$ we have
$\pi_{0}(dn,r) = 81$, but counting the gaps we see that
the genus of $L'$ is $g'=2\cdot 13+x>81$ for $x > 55$, proving that $L'$ is not Weierstrass.

This is not optimal: Taking $dn = 38$ it can be improved to  $g'> 80$ for $x> 54$.
\end{example}
\begin{remark}
 
\begin{enumerate}
\item There exist non-Weierstrass semigroups of multiplicity $m$ for every $m \ge 6$. Examples generalizing Buchweitz's construction provide non-Weierstrass semigroups of multiplicity $m$ and genus $m+3$ for every $m\ge 13$ \cite[Section 2]{MR1634703}. The degree-special semigroups of multiplicity $m \in \{6,\ldots,12\}$ of smallest genus are shown in Table 3.

\begin{table}[h]\label{table2}
\centering
\begin{tabular}{ccc}
\toprule
Multiplicity & Generators & Genus \\
\midrule
6  & $\{6, 9, 13, 16\}$  & 13 \\
7  & $\{7, 15, 16, 24\}$ & 22 \\
8  & $\{8, 9, 12, 13\}$  & 13 \\
9  & $\{9, 10, 14, 15\}$ & 17 \\
10 & $\{10, 11, 15, 16\}$ & 21 \\
11 & $\{11, 12, 17, 18\}$ & 26 \\
12 & $\{12, 13, 15, 29\}$ & 27 \\
12 & $\{12, 14, 15, 25\}$ & 27 \\
\bottomrule
\end{tabular}
\caption{Degree-special semigroups of multiplicity $m \in \{6, \ldots, 12\}$ 
of smallest genus, with non-Weierstrass semigroups of multiplicity 
$m \geq 13$ supplied by Buchweitz's construction.}
\end{table}

\medbreak
 \item Kunz and Waldi \cite{MR3621681} studied a family of semigroups $S$ such that when $S$ is minimally generated
 by 4 elements the resolution of the semigroup algebra has format 
 $\{1,6,8,3\}$. Moreover the ideal of a 4-generator Kunz-Waldi semigroup can be written as a sum of 3 perfect, 3-generator ideals \cite[Sect. 4.1]{singh2024}. The special semigroup with generators $\{6,9,13,16\}$ is Kunz-Waldi, 
 but the special semigroup $\{6, 9, 14,17\}$, for example, is not.
 
\end{enumerate}

\end{remark}

\section{Semigroups of genus $<13$}\label{comp}

To show that all 1413 semigroups of genus $<13$ are Weierstrass, we used
Pinkham's Theorem  \cite{MR376672} and 
 computed explicit one-parameter positively graded smoothing families  over $\QQ$ directions for each of the 632 examples that was not already known. This shows that they are Weierstrass semigroups over $\QQ$, and thus over $\CC$. 

We exhibit these families in the auxiliary file to the arXiv preprint of this paper  
or in the GitHub repository {\bf eisenbud/WeierstrassSemigroups} \url{https://github.com/eisenbud/WeierstrassSemigroups}. \\
In these places we also provide a {\it Macaulay2} program  {\bf 
allSemigroupsOfGe\-nusLessThan13AreWeierstrass.m2} that checks  flatness and smoothness for each of them. To run the code takes about 10 minutes.

Here is an outline of the computation we made to find a smoothing family for each semigroup. We begin the constructions over a finite field $\ZZ/p$ for computational efficiency. To aid in the eventual lifting to characteristic 0 we took $p$ on the order of $10^{7}$.

\begin{enumerate}
\item We form the graded universal unfolding of the semigroup ideal, adding only positive degree parameters.

An alternate approach is to use Nathan Ilten's {\it Macaulay2} package VersalDeformations \cite{VersalDeformationsSource, VersalDeformationsArticle}

Restricting to deformation parameters in a certain range of degrees, for example degrees $>b$, can make the computation
easier, and we were able to do this in most cases.

\item We use Gr\"obner bases to compute equations on the parameters that guarantee the constancy
of the Gr\"obner basis in the family, and thus flatness. Once such a family is found we prune the family
by eliminating the parameters
that are polynomial functions of other parameters, c.f. \cite{MR713093}.
\item Decompose the  base of the flat family into irreducible components and prune 
the family over each component again.
\item Pick a random point in each component and check the smoothness of the fiber over that point.

\item Checking smoothness of a fiber: If the set of generators $L$ of the semigroup has $n$ elements, then the affine curve $\widetilde C$ lies in $\AA^{n}$ and may have large codimension.  Hence computing all the minors of order $(n-1)$ of the Jacobian matrix $M$ of the ideal $I(\widetilde C)$ is often infeasible. We exploit the fact that the fibers are  1-dimensional as follows:

\begin{enumerate}
\item We find an ideal $J$ generated by some of the minors of order $n-1$ of $M$ that does not vanish identically on $\widetilde C$. 
For this, we need to find  minors $F$ not vanishing identically on $\widetilde C$.
To do this we look at the easier problem of finding complete intersections  $n-1$ homogeneous equations in the ideal of the semigroup. The minors of the corresponding equations of $\widetilde C$ often give us the  minors we need.

\item At this point, $\widetilde C$ is smooth outside the reduced points where $J$ vanishes. 
Represent the radical of  $J + I(\widetilde C)$
as an intersection of maximal ideals corresponding to points on $\widetilde C$.
 It now suffices to compute the rank of the Jacobian matrix at each of the 
  residue fields at these points, and computing the rank of a matrix over a field is easy.

\end{enumerate}

\item \label{6} Once we have found a flat smoothing family in characteristic $p$ over an irreducible component $B$ of the base, we try to lift it to a flat one-parameter family defined over $\QQ$. Any such lifting will be a smoothing as required.

\begin{enumerate}
 \item We try to find a point in $B$ with coordinates representable by integers of small absolute value. 
 \begin{enumerate}
     \item If the ideal of  $B$ is 0, we choose a point with small coefficients arbitrarily.
     \item We look for a subspace defined by $\dim B$ variables  whose common zero locus is a subspace $V$ of $B$.
    \item If the corresponding projection  map
    $$ \pi_{V}\colon B\setminus V \to \AA^{\dim B}$$ 
     is birational, then we can find a point $b\in B$ as the preimage of a point with small coordinates
    in $\AA^{\dim B}$.
  
   \item We then consider the fiber $C$ over $b$.
   \end{enumerate}
 \item (Clearing denominators) We multiply each equation of $C$ by a small integer and check whether, after multiplying, the coefficients can be represented by  integers of moderate size. If so we move these  equations to a polynomial ring over $\QQ$ and homogenize with an additional variable $z$.
 \item We then check whether the resulting family over $\Spec \QQ[z]$ is flat.
 \item If this is the case check whether the fiber over $z=1$ is a smooth curve.    
 \end{enumerate}
\end{enumerate}

The bottlenecks in this computation are computing  the flattening relations of a (restricted) unfolding, pruning its ideal, decomposing the base into  components; and checking smoothness. 

\begin{remark}
Step (\ref{6}.a) is  successful amazingly often. In fact for every one of the 632 semigroups of genus $<13$ that were not previously known to be Weierstrass with $B\neq 0$, we were able to find a subspace $V$, possibly after varying the base component $B$, and the choice of $b$. 

 A possible reason for this success is that the base of the deformation space is defined by a homogeneous ideal in a polynomial ring where  the variables have many different degrees. Quite often the largest degree variables occur only once and in quadratic terms of the equations, so that their values are determined by the values of the lower degree variables.
The fact that they often occur in quadratic monomials might come from the fact that the obstruction map is quadratic.
\end{remark}

\bibliographystyle{alpha}                  
\bibliography{MinimalNonWeierstrass}
\end{document}